\documentclass[graybox]{SNmult}

\usepackage[utf8]{inputenc}
\usepackage{type1cm}        %
\usepackage{makeidx}         %
\usepackage{graphicx}        %
\usepackage{multicol}        %
\usepackage[bottom]{footmisc}%
\usepackage[numbers]{natbib}

\usepackage{newtxtext}       %
\usepackage[varvw]{newtxmath}       %

\makeindex             %

\usepackage{mathtools}
\usepackage{amsmath}
\DeclareMathAlphabet{\mathpzc}{OT1}{pzc}{m}{it}
\usepackage{calc} %

\usepackage{siunitx}
\usepackage[hidelinks]{hyperref}
\usepackage{cleveref}
\usepackage{xcolor}
\usepackage{xargs}
\usepackage{soul}
\usepackage{subcaption}
\usepackage{nicefrac}
\usepackage{macros}
\usepackage[super]{nth}
\usepackage{csquotes}
\usepackage{dsfont}

\usepackage[mode=buildnew,subpreambles=true]{standalone}
\usepackage{shellesc}
\usepackage{tikz,tikzscale}
\usetikzlibrary{arrows,calc,fit,patterns,positioning,spy,backgrounds,decorations.fractals,shapes.misc}
\tikzset{boximg/.style={remember picture,black,thick,draw,inner sep=0pt,outer sep=0pt}}
\usepackage{pgf}
\usepackage{pgfplots}
\usepgfplotslibrary{statistics}
\usepgfplotslibrary{fillbetween}
\usepgfplotslibrary{colormaps}
\usepgfplotslibrary{groupplots,dateplot}
\pgfplotsset{compat=1.16}
\pgfplotsset{lua backend=true}

\pgfplotsset{
  	/pgfplots/colormap={redtoblue}{
			rgb=(
			0.12548999999999999,
			0.26274500000000001,
			0.50196099999999999)
			rgb=(
			0.57254899999999997,
			0.819608,
			0.87843099999999996)
			rgb=(
			0.94117600000000001,
			0.98431400000000002,
			0.98823499999999997)
			rgb=(
			0.76078400000000002,
			0.29411799999999999,
			0.21176500000000001)
			rgb=(
			0.54901999999999995,
			0.066667000000000004,
			0.098039000000000001)
    }
}

\begin{document}
\title*{h-Adaptive FV Subcell Shock-Capturing for DGSEM on Heterogeneous Curvilinear Meshes}
\titlerunning{Shock-Capturing for DGSEM on Heterogeneous Curvilinear Meshes}
\author{Anna Schwarz$^{\dag,\ast}$\orcidID{0000-0002-3181-8230}, Jens Keim$^{\dag,\ast}$\orcidID{0000-0002-2338-1497}, Christian Rohde$^\ddag$\orcidID{0000-0001-9183-5094} and Andrea Beck$^\dag$\orcidID{0000-0003-3634-7447}}
\institute{$^\ast$ J. Keim and A. Schwarz share first authorship. \at $^\dag$ Institute of Aerodynamics and Gas Dynamics, University of Stuttgart, Wankelstr. 3, 70563 Stuttgart, \email{schwarz/keim/beck@iag.uni-stuttgart.de}
\and $^\ddag$ Institute of Applied Analysis and Numerical Simulation, University of Stuttgart, Pfaffenwaldring 57, 70569
Stuttgart, \email{crohde@mathematik.uni-stuttgart.de}}

\maketitle
\abstract*{
High-order methods offer superior dispersion and dissipation properties compared to low-order schemes but require robust stabilization for discontinuities. To ensure stability, local artificial viscosity is common, but often degrades sub-element resolution. Conversely, subcell resolution preserving limiting strategies such as the finite volume subcell method are typically restricted to uniform topologies, such as purely hexahedral or simplex meshes, or linear elements. This leaves a significant gap in treating the hybrid-element topologies necessary for complex engineering geometries. To bridge this gap, we introduce a robust shock-capturing approach for the discontinuous Galerkin spectral element method on mixed curvilinear meshes containing hexahedral, prismatic, tetrahedral, and pyramid elements. Non-hexahedral elements are handled via collapsed coordinate transformations. The proposed method utilizes an $h$-adaptive finite volume subcell scheme with an arbitrary subcell resolution up to $2\ppn+1$. Special care is taken to ensure conforming subcell distributions across different element types. Crucially, discrete conservation is proven theoretically and verified numerically, alongside validations for high-order convergence and robust shock capturing. Finally, the method's applicability to complex configurations is demonstrated through a simulation of the flow around a NACA 0012 airfoil.
}

\abstract{
High-order methods offer superior dispersion and dissipation properties compared to low-order schemes but require robust stabilization for discontinuities. To ensure stability, local artificial viscosity is common, but often degrades sub-element resolution. Conversely, subcell resolution preserving limiting strategies such as the finite volume subcell method are typically restricted to uniform topologies, such as purely hexahedral or simplex meshes, or linear elements. This leaves a significant gap in treating the hybrid-element topologies necessary for complex engineering geometries. To bridge this gap, we introduce a robust shock-capturing approach for the discontinuous Galerkin spectral element method on mixed curvilinear meshes containing hexahedral, prismatic, tetrahedral, and pyramid elements. Non-hexahedral elements are handled via collapsed coordinate transformations. The proposed method utilizes an $h$-adaptive finite volume subcell scheme with an arbitrary subcell resolution up to $2\ppn+1$. Special care is taken to ensure conforming subcell distributions across different element types. Crucially, discrete conservation is proven theoretically and verified numerically, alongside validations for high-order convergence and robust shock capturing. Finally, the method's applicability to complex configurations is demonstrated through a simulation of the flow around a NACA 0012 airfoil.
}

\section{Introduction}
\label{sec:introduction}

High-order numerical methods gained significant interest due to their inherent geometric flexibility and superior dispersion and
dissipation characteristics compared to lower-order schemes \cite{Ainsworth2004a}.
While these methods provide exceptional accuracy in smooth, multiscale regions of a flow, they are susceptible to Gibbs oscillations
at discontinuities.
This is particularly problematic for non-linear hyperbolic conservation laws, which serve as a foundation for modeling diverse physical processes \cite{Dafermos2010}.
An intrinsic property of such equations is the potential for classical solutions to break down, leading to the formation of discontinuities even when the initial data are smooth.
In the context of the Euler equations, these features typically manifest as linear contact discontinuities or non-linear shock waves, both of which necessitate robust stabilization to maintain numerical integrity.

Numerical stabilization in the presence of discontinuities remains a fundamental challenge for high-order schemes.
First, the effective handling of discontinuities is predicated on accurate shock detection, a non-trivial task addressed via parameter-dependent indicator functions.
These are broadly categorized into a priori approaches, derived from physical heuristics \cite{Jameson1981}, modal decay estimates
\cite{Huerta2012, Klockner2011}, or image-detection algorithms \cite{Beck2020}, and a posteriori strategies, such as the MOOD limiter \cite{Clain2011, Dumbser2014}, invariant domain preserving schemes~\cite{Guermond2016} or data-driven approaches~\cite{Schwarz2022a}.
While MOOD-based methods guarantee less oscillatory solutions by recomputing segments that violate admissibility criteria, they often incur higher computational costs.
Second, the necessity for stabilization is rooted in Godunov’s Theorem \cite{Godunov1959}, which states that linear monotone schemes are limited to first-order accuracy, thereby requiring non-linear limiting to resolve shocks without spurious oscillations.
To this end, methods utilizing total variation diminishing (TVD) limiting strategies \cite{Leveque2002, Toro2009} and (weighted) essentially non-oscillatory ((W)ENO) reconstructions \cite{HartenENO, weno} have become standard in the literature.
For the class of methods including flux reconstruction, spectral difference, and discontinuous Galerkin (DG), which constitute
the primary focus of this work, a common stabilization ansatz involves the local introduction of artificial viscosity.
This regularization can be realized through various techniques, such as local solution filtering \cite{Hesthaven2008}, the
application of a local diffusion operator \cite{Jameson1981,persson2006}, flux correction methods \cite{Kuzmin2005,Vilar2019a}, and hybrid subcell strategies.
These hybrid approaches, such as convex blending \cite{Hennemann2021} or $h$-refined low-order subcell solutions based on TVD \cite{Sonntag2017, Mossier2022} or WENO \cite{Dumbser2008b, Dumbser2014, Dumbser2016} operators, represent a robust path toward maintaining high-order precision alongside shock stability.

While local diffusion operators are a common strategy for stabilization, they inherently degrade sub-element resolution, often leading to excessive numerical dissipation.
More sophisticated approaches, such as finite volume (FV) subcell methods, offer superior accuracy by preserving sub-grid features.
However, their application in the current literature remains largely bifurcated between purely hexahedral or purely simplex
discretizations. While recent advancements have extended these techniques to Voronoi or general polygonal meshes~\cite{Busto2020},
such implementations are predominantly restricted to the two-dimensional space and linear elements. Furthermore, they often necessitate complex quadrature rules or the internal subdivision of cells into simplex elements, which can significantly increase computational overhead and implementation complexity.
This leaves a significant gap regarding the treatment of hybrid-element topologies, which are essential for resolving complex engineering geometries where mixed-mesh flexibility is required \cite{Shepherd2008}.

To address this limitation, this paper presents a robust and efficient shock-capturing approach based on an $h$-adaptive FV subcell scheme for the discontinuous Galerkin spectral element method (DGSEM).
The core novelty of this work lies in achieving provably conservative shock capturing across mixed, curvilinear meshes composed of hexahedral, prismatic, tetrahedral, and pyramid elements, thereby addressing a long-standing challenge in high-order methods for complex industrial geometries.
The DGSEM is highly efficient on modern hardware~\cite{Blind2024,Kempf2024,Schwarz2025,Kopper2025a}, offering high-order accuracy in smooth regions along with favorable dispersion and dissipation properties. While traditionally leveraging tensor-product operators on hexahedral elements \cite{Orszag1979}, the method has been extended to non-hexahedral and curved elements via collapsed coordinate transformations, such as the Duffy transformation \cite{Chan2016, Duffy1982, Montoya2024a}.
Our approach builds upon the a priori FV subcell limiting proposed by \cite{Sonntag2017}, while ensuring that the treatment of interfaces between differing element types remains conforming.
Specifically, we extend the $h$-adaptive subcell resolution strategy of \cite{Dumbser2014,Mossier2022,Mossier2026} to curvilinear hybrid meshes and provide the analytical and numerical proof of its discrete conservation across these mixed interfaces.
With this, the scheme ensures strict conservation, high-order accuracy, and stability across the diverse element types and hybrid interfaces required for complex engineering problems.

The outline of this paper is as follows: First, the governing equations and their numerical treatment, including the shock capturing
approach, are outlined in~\cref{sec:methods}. This is followed by a thorough validation and testing of the resulting
scheme in~\cref{sec:validation}. The paper is finalized by a brief conclusion in~\cref{sec:conclusion}.

\section{Numerical Methods}
\label{sec:methods}
In the following, the numerical treatment of the governing equations is discussed, including the specific formulations for the DGSEM, the identification of troubled cells, and the subsequent shock-capturing procedure.

\subsection{Notation}
To maintain a clear mathematical hierarchy between the DG element and the FV subcell level, a consistent index convention is adapted throughout this work.
Volume-dependent quantities bound to a DG element $\iel$ are designated by the superscript $(\cdot)^{(\iel)}$, whereas data belonging to a face $\ifa$ of that element are denoted by the subscript $(\cdot)_{\ifa}$.
When moving to the subcell level, quantities belonging to an individual FV subcell $\ielFV$ inside the DG element $\iel$ are identified via the superscript $(\cdot)^{(\ielFV)}$.
Finally, to prevent cluttered expressions when evaluating interface metrics and operators, surface data isolated on a local subcell face $\ifaFV$ are tracked via $(\cdot)_\ifaFV$.

\subsection{Governing equations}

The governing equations considered in this work are the compressible Navier--Stokes--Fourier equations, written as
\begin{align}
  \lr{\cons}_t + \gradientX \cdot \fphys(\cons,\nabla \cons) = \Null, \hspace{0.5cm} \ \fphys(\cons,\nabla \cons) = \Fc - \Fv,
  \label{eq:theory:euler}
\end{align}
where $\cons=\arr{\rho, \rho \vel, \rho e} \in \R^{\nvar}$, $\nvar=d+2$, is the vector of conserved variables, consisting of the density $\rho$, the velocity vector $\vel=\arr{\vel[1],...,\vel[d]} \in \R^d$ and the total energy $e$ per unit mass, $d$ is the number of dimensions, and $\fphys \in \R^{\nvar\times d}$ are the physical fluxes, comprised of the convective flux $\Fc$ and the viscous flux $\Fv$, given as
\begin{align}
  \Fc =& \arr{
    \makebox[\widthof{$\rho \vel$}][l]{$\rho \vel$},
    \makebox[\widthof{$\rho \vel \otimes \vel + \p \mathbf{I}$}][l]{$\rho \vel \otimes \vel + \p \mathbf{I}$},
    \makebox[\widthof{$\stress \cdot \vel + \conductivity \nabla T$}][l]{$(\rho e + \p) \vel$}
    }, \\
  \Fv =& \arr{
    \makebox[\widthof{$\rho \vel$}][l]{$0$},
    \makebox[\widthof{$\rho \vel \otimes \vel + \p \mathbf{I}$}][l]{$\stress$},
    \makebox[\widthof{$\stress \cdot \vel + \conductivity \nabla T$}][l]{$\stress \cdot \vel + \conductivity \nabla T$}
    },
\end{align}
with the unit tensor $\mathbf{I} \in \R^{d \times d}$, the temperature $T$, the thermal conductivity $\conductivity$, and the pressure for a perfect gas $p=(\gamma-1)(\rho e- 0.5 \rho (\vel \cdot \vel))$ with $\gamma=1.4$.
Following Stokes' hypothesis, which assumes that the bulk viscosity is zero, the viscous stress tensor reduces to $\stress =
\dynvisc (\lr{\nabla \vel}^\transpose + \nabla \vel - \frac{2}{3} \lr{\nabla \cdot \vel}\mathbf{I}) \in \R^{d \times d}$ for a Newtonian fluid, where $\dynvisc$ denotes the dynamic viscosity.
The heat flux is modeled according to Fourier's law with $\conductivity = \frac{c_p \dynvisc}{\Pr}$, the Prandtl number $\Pr$, and the specific heat at constant pressure $c_p$ of ambient air.

\subsection{Discontinuous Galerkin Spectral Element Method}
Following the method of lines approach, the conservation equations are discretized in space by the discontinuous Galerkin spectral element method.
The computational domain $\Omega \subset \Rd$ is composed of $\nel$ non-overlapping, conforming curvilinear polytopal elements $\{\Omega^{(\iel)}\}_{\iel \in \{1:\nel\}}$.
Curvilinear elements are approximated by normalized orthonormal Legendre or Proriol--Koornwinder--Dubiner (PKD) polynomials of degree $\ppngeo$~\cite{Keim2026}.
The boundary of each computational element $\partial \Omega^{(\iel)}$ is partitioned into faces $\{\Gamma^{\ifa}\}_{\smash{\ifa \in
  \{1:\Nf\}}}$, $\Gamma^{\ifa} \subset \partial \Omega^{(\iel)}$ with an outward pointing unit normal vector $\normalvec$, and $\Nf$ denotes the number of element faces of an element $(\iel)$.
Each element is transformed to the polytopal reference space by the inverse of the mapping $\mapping^{(\iel)} : \refraumcol \to \Omega^{(\iel)}, \refpos
\mapsto \mathbf{x}$ from the
reference element $\refpos \in \refraumcol$ to the physical space $\mathbf{x} \in \Omega$, with the corresponding Jacobian matrix
$\Jm^{(\iel)} = \lr{\gradientXI \mapping^{(\iel)}}$. The diagonal matrix of its determinant is defined as $\diagJ^{(\iel)}=\diag(\det\Jm^{(\iel)})$.
Subsequently, each polytopal reference element is mapped to the unit hexahedral $\refraum$ via the inverse of the mapping
$\mappingcol : \refraum = [-1,1]^d \to \refraumcol, \refposquad \mapsto \refpos$ and projected onto the $L_2$-space comprised of polynomials
$\testfunc \in \P(\refraum,\R^{\nvar})$ of order $\ppn$, which are chosen to be the tensor-product of 1D Lagrange polynomials
$\testfunc_{ijk}=\lagrange_i(\refposquad[1])\lagrange_j(\refposquad[2])\lagrange_k(\refposquad[3])$ with $i,j,k=0,\ldots,\ppn$.
The Jacobian matrix of this mapping is defined as $\Jmref= \lr{\gradientXI \mappingcol}$ with $\diagJref=\diag(\det\Jmref)$.
Exploiting the Galerkin property, the nodal solution in the reference element $\refraum$ is approximated by 1D Lagrange polynomials of degree $\ppn$.

The semi-discrete weak form is obtained using Legendre--Gauss quadrature on $\nq$ Legendre--Gauss nodes, where the interpolation and
integration nodes are collocated, leading to
\begin{align}
  \Mm \diagJref \diagJ^{(\iel)} \cons_t^{(\iel)} = & \ \mathcal{R}(u) \nonumber \\
                              \mathcal{R}(u) =& \sum_{p=1}^{d} \Dm^\transpose \Mm \fphysref^{(\iel,p)}(\cons)
                                              - \sum_{\ifa=1}^{\Nf} \Vm_{\ifa}^\transpose \Wm \diagJref_\ifa \diagJ^{(\iel)}_\ifa f^\ast(\cons_{\ifa}^+ , \cons_{\ifa}^-; \normalvec_{\ifa}),
  \label{eq:dgsem:convective}
\end{align}
with the element-local vector of degrees of freedom $\cons^{(\iel)}$, the diagonal mass matrix $\Mm$, comprised of the volume quadrature weights, the polynomial derivative matrix $\Dm$, and the surface quadrature weights $\Wm$.
The outward-pointing physical normal vector $\normalvec \in \Rd$ is computed via Nanson's formula
\begin{align}
  \abs{\adjJ^{(\iel)} \normalvecref_\ifa} \normalvec_\ifa = \adjJ^{(\iel)} \normalvecref_\ifa    \label{eq:nanson}
\end{align}
with the transpose of the adjoint of the Jacobian $\adjJ^{(\iel)} = (\adj {\Jm^{(\iel)}})^\transpose$, the surface element $\J^{(\iel)}_\ifa = \abs{\adjJ^{(\iel)} \normalvecref_\ifa}$, $\smash{\diagJ^{(\iel)}_\ifa = \diag(\J^{(\iel)}_\ifa)}$, and the unit normal vector of the corresponding side in the polytopal reference space $\normalvecref$.
The contravariant convective fluxes in the volume are denoted by $\fphysref^{(\iel)}$.
The numerical flux is denoted by $f^\ast$ and, in this work, approximated via Roe's numerical flux with the entropy fix by~\cite{Harten1983b} or by
the Rusanov flux function.
The approximate solutions on the left $\cons_\ifa^{+}$ and right $\cons_\ifa^{-}$ of a face $\ifa$ are obtained via 1D
  volume operations, owing to the tensor-product structure of the interpolation, resulting in
\begin{align}
  {\cons}_{\ifa}^+ = \Vm_{\ifa} {\cons}^{(\iel),+}, ~ (\Vm_{\ifa})_{ij} = \lagrange_j(\refposquad_{\ifa,i}), %
\end{align}
with the quadrature nodes on the respective facets $\refposquad_{\ifa}$.

In this work, the BR1 (lifting) scheme~\cite{Bassi1997} is utilized to approximate the gradients of $\prim=\cons^{\text{prim}}$ in the
Navier--Stokes--Fourier equations, leading to
\begin{align}
  \Mm \diagJref \diagJ^{(\iel)} \cons_t^{(\iel)} =& \ \mathcal{R}(u) \\
                              - & \sum\limits_{p=1}^{d} \Dm^\transpose \Mm \fphysref_v^{(\iel,p)}(\cons, \mathbf{g}) \nonumber
                              + \sum_{\ifa=1}^{\Nf} \Vm_{\ifa}^\transpose \Wm \diagJref_{\ifa}
                                \diagJ^{(\iel)}_\ifa \avg{\fphys_v(\cons_\ifa,\nabla \prim_\ifa) \cdot \normalvec_\ifa}, \nonumber
  \label{eq:dgsem:viscous}
\end{align}
with the lifted gradients of the primitive variables $\mathbf{g} = \gradientX \prim$, the contravariant viscous fluxes $\fphysref_v$ in the volume, and $\avg{\fphys_v(\cons,\nabla \prim) \cdot \normalvec}$ are the average viscous fluxes at an element interface.
The semi-discrete DGSEM is advanced in time by a low-storage explicit high-order accurate Runge--Kutta method using the modal polynomial coefficients instead of the nodal degrees of freedom to avoid the excessive time step restriction due to the collapsing~\cite{Montoya2024b}.
Further details on the scheme can be found in~\cite{Keim2026}.

\subsection{\label{sec:methods:shock_caputring}Shock Capturing}

In this work, the shock capturing is based on a switching of troubled DG cells to a second-order $h$-adaptive total variation
diminishing FV subcell operator.
The FV subcell scheme is achieved by partitioning the DG reference element $\refraumcol$ into a set of $N_\ielFV$ subcells, $\{E_\ielFV \subset \refraumcol | \ielFV=1,...,N_\ielFV\}$,
of the same element type, if possible, allowing for an h-adaptive FV subcell scheme.
While the subcell resolution, $N_\ielFV$, may be selected arbitrarily, a minimum of $N_\ielFV \geq (\ppn+1)^d$ subcells is necessary to maintain robust and accurate shock capturing~\cite{Dumbser2016}.
To avoid a time-step restriction more stringent than that of the DG method, the upper bound for an equidistant distribution is typically constrained to $N_\ielFV \leq (2\ppn+1)^d$ when employing Gauss--Legendre nodes.
As such, the subcell grid resolution $\FVppn$, representing the number of 1D subcell subdivisions along an element edge, is bounded by $\ppn+1 \leq \FVppn \leq 2\ppn+1$.

In the following, the FV subcell connectivity for all considered element types is established. Subsequently, the conservative coupling between the FV subcells and the DG elements is outlined, culminating in the formulation of the complete hybrid scheme. Finally, the troubled-cell detection algorithm utilized to dynamically switch between the two discretization operators is briefly presented.

\subsubsection{Finite Volume Subcell Connectivity}
The set $S_{i,j,k}$ defines the unique coordinates of the corner nodes $\refpos$ that bound a single FV subcell inside a reference element $\refraumcol$.
For a hexahedral element, the corner node connectivity of a subcell at index $(i,j,k)$ is explicitly given by
\begin{align}
  S^\text{hexa}_{i,j,k} = \{\refpos_{i,j,k},\refpos_{i+1,j,k},\refpos_{i,j+1,k},\refpos_{i+1,j+1,k},\refpos_{i,j,k+1},\refpos_{i+1,j,k+1},\refpos_{i,j+1,k+1},\refpos_{i+1,j+1,k+1}\}
\end{align}
with $i,j,k = 0,...,\FVppn$ (the corresponding 2D subgrid elements for a quadrilateral element are retrieved by dropping the last four entries).
For the prism, the corner node connectivity of the two subcells at index $(i,j,k)$ is written as
\begin{align}
  S^{\text{prism},U}_{i,j,k} =& \{\refpos_{i,j,k},\refpos_{i+1,j,k},\refpos_{i,j+1,k},\refpos_{i,j,k+1},\refpos_{i+1,j,k+1},\refpos_{i,j+1,k+1}\},\nonumber \\
  S^{\text{prism},L}_{i,j,k} =&
  \{\refpos_{i+1,j+1,k},\refpos_{i+1,j,k},\refpos_{i,j+1,k},\refpos_{i+1,j+1,k+1},\refpos_{i+1,j,k+1},\refpos_{i,j+1,k+1}\},
\end{align}
with $0 \leq k \leq \FVppn$, $0 \leq j \leq \FVppn$, and $0 \leq i \leq \FVppn - j$ (the corresponding 2D subgrid connectivity for a triangular element is retrieved by dropping the last three entries).
The upper and lower prisms are expressed as $S^{\text{prism},U}$ and $S^{\text{prism},L}$, respectively.
The subcell distribution in the tetrahedron is chosen according to~\cite{Dumbser2016}, which thoroughly reported the corresponding 3D connectivity.

Unlike the previous elements types, a standard uniform grid subdivision using the same element type cannot be straightforwardly applied to a pyramid without a severe loss of resolution or geometric conformity. While several topological splitting strategies exist, this work subdivides the reference pyramid into a combination of smaller pyramidal and intermediate tetrahedral subcells. This specific decomposition ensures a nearly equidistant subcell distribution while avoiding inter-element hanging nodes, thereby preserving strict grid conformity without introducing an additional timestep restriction.
Consequently, the definition of the 3D subgrid connectivity within a reference pyramid is inherently more complex. The subcells are organized into horizontal cross-sectional layers indexed by the vertical coordinate $k$. As the pyramid converges toward its apex, the horizontal grid dimensions shrink proportionally, bounding the directional indices by $0 \leq k \leq \FVppn$, $0 \leq j \leq \FVppn - k$, and $0 \leq i \leq \FVppn - k$. Depending on the local orientation and geometric placement within a given layer $k$, the connectivity sets for the subcells are defined as
\begin{align}
  S^{\text{pyra},U}_{i,j,k} =& \{\refpos_{i,j,k},\refpos_{i+1,j,k},\refpos_{i+1,j+1,k},\refpos_{i,j+1,k},\refpos_{i,j,k+1}\},\nonumber \\
  S^{\text{pyra},L}_{i,j,k} =& \{\refpos_{i,j,k+1},\refpos_{i+1,j,k+1},\refpos_{i+1,j+1,k+1},\refpos_{i,j+1,k+1},\refpos_{i+1,j,k}\},\nonumber \\
  S^{\text{pyra},I}_{i,j,k} =& \{\refpos_{i,j+1,k},\refpos_{i,j+1,k+1},\refpos_{i,j,k+1},\refpos_{i+1,j+1,k}\},\nonumber \\
  S^{\text{pyra},II}_{i,j,k} =& \{\refpos_{i+1,j,k},\refpos_{i+1,j,k+1},\refpos_{i,j,k+1},\refpos_{i+1,j+1,k}\},
\end{align}
where $S^{\text{pyra},U}$ and $S^{\text{pyra},L}$ represent the upward-pointing and downward-pointing pyramidal subcells, respectively, while $S^{\text{pyra},I}$ and $S^{\text{pyra},II}$ denote the intermediate tetrahedral subcells that close the gaps between them.

\begin{figure}
  \includegraphics[width=\linewidth]{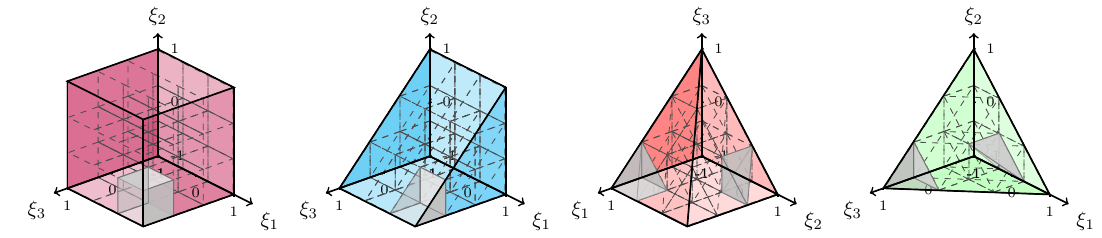}
  \caption{Exemplary FV subcell distribution in the polytopal reference space for the hexahedron, prism, pyramid, and tetrahedron
    using $\FVppn=3$ (from left to right).}
\end{figure}

\subsubsection{h-Adaptive Finite Volume Discretization}
The second-order $h$-adaptive FV subcell scheme on curvilinear mixed elements is obtained by an $L_2$ projection of the governing equations
onto the space of piecewise constant test functions. The semi-discrete scheme in one subcell $\ielFV$ is written as
\begin{align}
  \diagJ^{(\ielFV)} \cons^{(\ielFV)}_t
  = - \frac{1}{|E_{\ielFV}|} \sum_{\ifaFV=1}^{N_\ifaFV} \diagJ^{(\ielFV)}_\ifaFV
  \left(f^\ast(\cons_\ifaFV^{+},\cons_\ifaFV^{-};\normalvec_\ifaFV)
    - \fphys_v(\cons_\ifaFV,\nabla \prim_\ifaFV) \cdot \normalvec_\ifaFV \right),
  \label{eq:FV_scheme}
\end{align}
where $|E_\ifaFV|$ represents the reference volume of the $\ifaFV$-th subcell.
The corresponding integral mean value is denoted as $\cons^{(\ielFV)}$, the integral mean of the determinant of the Jacobian of this subcell is $\diagJ^{(\ielFV)}$, the surface element of one FV subcell face is $\diagJ^{(\ielFV)}_\ifaFV$, and $N_\ifaFV$ is the number of FV subcell sides for element $\ielFV$. Both Jacobians have to be defined such that the coupling between the FV subcells and the DG scheme is conservative, which will be thoroughly discussed in~\cref{sec:FVcoupling}.
The conservative variables left and right to an FV subcell face, $\cons_\ifaFV^+$ and $\cons_\ifaFV^-$, are obtained via interpolation based on the reconstructed slopes.
The viscous fluxes $\fphys_v(\cons_\ifaFV,\nabla \prim_\ifaFV)$ are computed using the averaged adjacent integral mean values on the respective subcell face.
The gradients $\nabla \prim_\ifaFV$ are approximated using the modified correction approach proposed by~\cite{Haselbacher1999} based on the unlimited least squares gradients.%

A total-variation diminishing second-order FV operator is obtained via a piecewise linear reconstruction, followed by a nonlinear
flux limiter such as the generalized minmod function \cite{Leer1977, Kurganov2007}. This procedure yields the following limited slope
\begin{align*}
  \phi_{lim}(r;\beta) = \min_{\ifaFV} \left(\text{minmod} \left(\beta r_\ifaFV, 0.5(1 + r_\ifaFV),\beta \right) \right),
\end{align*}
with $\beta \in [1,2]$, i.e., the final slope is limited using a minmod-like approach over all adjacent neighboring elements. The expression $r_\ifaFV$ is the ratio of the slope computed by the integral mean values between the left and right side of a subcell face $\ifaFV$ and the unlimited least squares gradient in the corresponding direction, $\nabla \prim^{(\ielFV)}$, given as
\begin{align*}
  r_\ifaFV = \frac{(\prim^{(\ielFV_\text{neigh})} - \prim^{(\ielFV)}) / \abs{\mathbf{x}^{(\ielFV_\text{neigh})} - \mathbf{x}^{(\ielFV)}}}{(\nabla
    \prim^{(\ielFV)} \cdot (\mathbf{x}_\ifaFV - \mathbf{x}^{(\ielFV)})) / \abs{\mathbf{x}_\ifaFV - \mathbf{x}^{(\ielFV)}}}.
\end{align*}
Recall that the primitive variables of the integral means are $\prim$.
The barycenter of a FV subcell face is denoted as $\mathbf{x}_\ifaFV$, the barycenter of the $\ielFV$-th FV subcell is $\mathbf{x}^{(\ielFV)}$, and $(\cdot)^{(\ielFV_\text{neigh})}$ denotes quantities from the corresponding adjacent FV subcell.
The authors are aware that a vertex based limiter offers enhanced robustness, albeit at an elevated computational expense~\cite{Yoon2008}.
Subsequently, the conservative variables at the subcell faces for one subcell element $\ielFV$ are obtained via linear interpolation of
the solution using the limited gradients.

\subsubsection{Coupling Between FV Subcells and DG Cells}
\label{sec:FVcoupling}
To guarantee a conservative switching between DG cells and FV subcells, strict integral conservation must be enforced over the element domain $\refraumcol$
\begin{align}
  \int_{\refraumcol} (\diagJ \cons)[\refpos] \ d\refpos = \sum_{\ielFV=1}^{N_\ielFV} \int_{E_\ielFV} (\diagJ \cons)[\refpos] \ d\refpos = \sum_{\ielFV=1}^{N_\ielFV} \diagJ^{(\ielFV)} \cons^{(\ielFV)} |E_\ielFV|.
  \label{eq:fvdg_integralcons_vol}
\end{align}
Subsequent application of Gaussian quadrature on Gauss--Legendre nodes to evaluate the second integral yields the discrete conservation condition
\begin{align}
  \sum_{\ielFV=1}^{N_\ielFV} \diagJ^{(\ielFV)} \cons^{(\ielFV)} |E_\ielFV| = \sum_{\ielFV=1}^{N_\ielFV} \PDGtoFV_\ielFV ~ \diagJ^{(\iel)}
  \cons^{(\iel)}.
  \label{eq:fvdg_coupling_vol}
\end{align}
The operator $\PDGtoFV_\ielFV$ emerges as the row-wise volume projection matrix from the DG space to the $\ielFV$-th subcell.
Due to the non-tensor-product layout of the subcells, the mapping between the DG and FV subcell representations is managed via an orthogonal modal space spanned by normalized orthonormal Legendre or PKD polynomials $\varphi$, see~\cite{Keim2026} for more details, by leveraging the polynomial interpolation matrix $\Vm(\refpos) = \varphi(\refpos)$ and the nodal-to-modal projection operator $\Pm= \mathcal{M}^{-1} \Vm^\transpose(\refpos) \Mm$, leading to
\begin{align}
  \PDGtoFV_\ielFV = \frac{1}{|E_\ielFV|} \mathbf{1}^\transpose (\Mm^{(\ielFV)} \Vm(\refpos^{(\ielFV)}) \Pm),
  \label{eq:PDGtoFV}
\end{align}
with $\mathbf{1} \in \R^{N_\ielFV}$ as the vector of all ones, and $\refpos^{(\ielFV)}$ denotes the quadrature nodes utilized in the $\ielFV$-th subcell.
The constituent terms are defined as follows.
The weight-adjusted mass matrix is denoted as $\mathcal{M}=\Vm^\transpose \widetilde{\diagJ} \Mm \Vm$, the projected Jacobian is $\widetilde{\diagJ}$, see~\cite{Keim2026}, and the diagonal matrix $\Mm^{(\ielFV)}$ is comprised of the volume quadrature weights in the $\ielFV$-th subcell.
From~\cref{eq:fvdg_coupling_vol}, the subcell determinant of the Jacobian is consistently retrieved via
\begin{align}
  \diagJ^{(\ielFV)} = \PDGtoFVJ_\ielFV \diagJ^{(\iel)},
  \label{eq:fvdg_coupling_metricsvol}
\end{align}
enforcing combined with~\cref{eq:fvdg_coupling_vol} conservation between the DG and FV subcell solutions for curvilinear meshes.
The consistent row-wise DG-to-FV geometric projection operator $\PDGtoFVJ_\ielFV$ shares the identical structure of~\cref{eq:PDGtoFV}, with the weight-adjusted mass matrix reducing to its unweighted form, yielding
\begin{align}
  \PDGtoFVJ_\ielFV = \frac{1}{|E_\ielFV|} \mathbf{1}^\transpose (\Mm^{(\ielFV)} \Vm(\refpos^{(\ielFV)}) \widetilde{\Pm}, \ \text{with} \ \widetilde{\Pm} = (\Vm^\transpose \Mm \Vm)^{-1} \Vm^\transpose(\refpos) \Mm.
\end{align}

The inverse problem, reconstructing a continuous DG polynomial from $N_\ielFV$ subcell integral mean values, yields an overdetermined system whenever the subcell resolution satisfies $\FVppn \geq \ppn+1$. Following \cite{Dumbser2016}, this is resolved via a constrained least-squares formulation subject to the integral conservation constraint in \cref{eq:fvdg_integralcons_vol}, resulting in the following row-wise FV-to-DG projection operator, given as
\begin{align}
  \PFVtoDG_\ielFV = \frac{1}{|E_\ielFV|} \Vm(\refpos) \Vm^+_\ielFV,
  \label{eq:PFVtoDG}
\end{align}
where the pseudo-inverse matrix $\Vm^+_\ielFV$ is defined via the row-integrated Vandermonde vector $\check{\mathbf{v}}^{(\ielFV)} = \mathbf{1}^\transpose \Mm^{(\ielFV)} \Vm(\refpos^{(\ielFV)})$ as
\begin{align}
  \Vm^+_\ielFV = \check{\mathbf{v}}^{(\ielFV)}(\check{\mathbf{v}}^{(\ielFV)})^\transpose)^{-1}\check{\mathbf{v}}^{(\ielFV)}.
\end{align}

To preserve positivity and boundedness, the numerical surface fluxes between adjacent DG and FV elements are evaluated exclusively on the low-order FV subgrid distribution, requiring a conservative mapping of the DG solution on the shared face to the FV subcell faces.
The surface projection matrices $\PDGtoFV_\ifaFV$ and $\PFVtoDG_\ifaFV$ are derived analogously to their volume counterparts by enforcing integral conservation across the shared element face boundary, leading to
\begin{align}
  \int_{\Gamma^\ifa} (\diagJ^{(\iel)}_\ifa \cons_\ifa)[\Gamma] \ d\Gamma = \sum_{\ifaFV=1}^{n_\ifaFV} \diagJ^{(\ielFV)}_\ifaFV \cons_\ifaFV |\Gamma_\ifaFV| = \sum_{\ifaFV=1}^{n_\ifaFV} \PDGtoFV_\ifaFV \diagJ^{(\iel)}_\ifa \cons_\ifa
  \label{eq:fvdg_coupling_face}
\end{align}
where $n_\ifaFV$ is the total number of FV subcell faces partitioning a single macroscopic DG face $\ifa$.
The row-wise operator $\PDGtoFV_\ifaFV$ emerges as the surface projection matrix from the high-order DG face nodes to the barycenter of the $\ifaFV$-th subcell face, given as
\begin{align}
  \PDGtoFV_\ifaFV = \frac{1}{|\Gamma_\ifaFV|} \mathbf{1}^\transpose (\Wm_\ifaFV \Vm(\refpos_\ifaFV) \Pm_\ifa),
  \label{eq:PDGtoFVFace}
\end{align}
where $\Pm_\ifa = \mathcal{M}_\ifa^{-1} \Vm^\transpose(\refpos_\ifa) \Wm$ is the nodal-to-modal face operator.
Here, $\Vm = \varphi^{(d-1)}(\refpos_\ifa)$ denotes the Vandermonde matrix constructed using Legendre or PKD polynomial basis functions of dimension $d-1$, $\refpos_\ifaFV$ denotes the quadrature nodes on the $\ifaFV$-th subcell face, and $\Wm_\ifaFV$ contains the associated surface quadrature weights.
The consistent reference area of the subcell face is defined as $|\Gamma_\ifaFV|$.
The weight-adjusted mass matrix on the surface is denoted as $\smash{\mathcal{M}_\ifa=\Vm(\refpos_\ifa)^\transpose \widetilde{\diagJ}_\ifa^{(\iel)} \Wm \Vm(\refpos_\ifa)}$, where $\smash{\widetilde{\diagJ}_\ifa^{(\iel)}}$ represents the projected surface element following~\cite{Keim2026}.

Minimizing the residual of the integral conservation constraint in~\cref{eq:fvdg_coupling_face} via a constrained least-squares approach yields the projection matrix $\PFVtoDG_\ifaFV$.
This operator is subsequently deployed to map the updated numerical surface fluxes from the FV subcell faces back to the high-order DG face, given as
\begin{align}
  \PFVtoDG_\ifaFV = \frac{1}{|\Gamma_\ifaFV|} \Vm(\refpos_\ifa) \Vm^+_\ifaFV,
  \label{eq:PFVtoDGFace}
\end{align}
where the pseudo-inverse matrix $\Vm^+_\ifaFV$ is defined via the row-integrated Vandermonde vector $\check{\mathbf{v}}_\ifaFV = \mathbf{1}^\transpose \Wm_\ifaFV \Vm(\refpos_\ifaFV)$ as
\begin{align}
  \Vm^+_\ifaFV = (\check{\mathbf{v}}_\ifaFV \check{\mathbf{v}}_\ifaFV^\transpose)^{-1}\check{\mathbf{v}}_\ifaFV.
\end{align}

In a coupled DG/FV framework on curvilinear meshes, maintaining discrete conservation and free-stream preservation requires a consistent geometric treatment between the high-order polynomial element faces and the piecewise constant subcell finite volume faces. This geometric constraint is formalized in the following lemma.

\begin{lemma}[Consistent Surface Area of a DG/FV Interface]
  Let $\Gamma_\ifaFV$ denote a piecewise constant FV subcell face overlapping a high-order DG element face.
  Assume the surface quadrature rule is of at least degree $2\ppn$ to guarantee the exactness of the mass matrix.
  To ensure that the integrated surface area over a global DG face matches the sum of its FV subcell counterparts, the physical normal vector $\normalvec_\ifaFV$ and the consistent physical surface area element $\diagJ^{(\ielFV)}_\ifaFV$ must satisfy the discrete projection
  \begin{align}
    \diagJ^{(\ielFV)}_\ifaFV \normalvec_\ifaFV = (\PDGtoFV_{\iel\to\ifaFV} \ \adjJ^{(\iel)}) \normalvecref_\ifaFV,
    \label{eq:FVnormvec}
  \end{align}
  where the row-wise projection vector $\PDGtoFV_{\iel\to\ifaFV} \in \R^{1 \times \nq}$ from the nodal DG volume quantities to the FV subcell interface barycenter is defined as
  \begin{align}
    \PDGtoFV_{\iel\to\ifaFV} = \frac{1}{|\Gamma_\ifaFV|} \mathbf{1}^\transpose \Wm_\ifaFV \Vm(\refpos_\ifaFV) \widetilde{\Pm}.
  \end{align}
  The physical subcell surface area element is defined explicitly as the magnitude of the projected reference metric vector
  \begin{align}
    \diagJ^{(\ielFV)}_\ifaFV = \abs{(\PDGtoFV_{\iel\to\ifaFV} \ \adjJ^{(\iel)}) \normalvecref_\ifaFV}.
  \end{align}
  \label{lem:subcell_area}
\end{lemma}

\begin{proof}
  Demanding that the integrated surface reference element area over a global DG surface matches the sum of its FV counterparts over all overlapping subcells yields
  \begin{align*}
    \int_{\Gamma_\ifa} \adjJ^{(\iel)} \normalvecref_\ifa \ d \Gamma = \sum_{\ifaFV=1}^{n_\ifaFV} \int_{\Gamma_\ifaFV} (\adjJ^{(\iel)})[\Gamma] \ \normalvecref_\ifa \ d \Gamma = \sum_{\ifaFV=1}^{n_\ifaFV} \adjJ^{(\ielFV)} \normalvecref_\ifaFV |\Gamma_\ifaFV|.
  \end{align*}
  On a straight or curvilinear element face, the unit reference normal vector of the subcell face is identical to its DG counterpart, $\normalvecref_\ifaFV = \normalvecref_\ifa$.
  Expressing the transpose of the adjoint of the Jacobian using the polynomial interpolation matrix $\Vm(\Gamma)$ and modal projection operator $\widetilde{\Pm}$ results in
  \begin{align*}
    \int_{\Gamma_\ifaFV} (\adjJ^{(\iel)})[\Gamma] \ \normalvecref_\ifa \ d \Gamma =
    \int_{\Gamma_\ifaFV} (\Vm(\Gamma) \widetilde{\Pm} \ \adjJ^{(\iel)}) \normalvecref_\ifaFV  d\Gamma.
  \end{align*}
  Evaluating this subcell face integral numerically via a Gaussian quadrature using $\refpos_\ifaFV$ as discrete integration nodes gives
  \begin{align*}
    \int_{\Gamma_\ifaFV} (\Vm(\Gamma) \widetilde{\Pm} \ \adjJ^{(\iel)}) \normalvecref_\ifaFV  d\Gamma =
      (\mathbf{1}^\transpose \Wm_\ifaFV \Vm(\refpos_\ifaFV) \widetilde{\Pm} \ \adjJ^{(\iel)}) \normalvecref_\ifaFV.
  \end{align*}
  This must be the same as the integral mean value at the subcell face barycenter.
  To isolate the integral mean value at the subcell face barycenter, the expression is scaled by the reference surface area of the subcell face, $|\Gamma_\ifaFV|$, leading to
  \begin{align*}
    \adjJ^{(\ielFV)} \normalvecref_\ifaFV = (\PDGtoFV_{\iel\to\ifaFV} \ \adjJ^{(\iel)}) \normalvecref_\ifaFV, \ \PDGtoFV_{\iel\to\ifaFV} = \frac{1}{|\Gamma_\ifaFV|} \mathbf{1}^\transpose \Wm_\ifaFV \Vm(\refpos_\ifaFV) \widetilde{\Pm}.
  \end{align*}
  This allows us to define a consistent, row-wise projection matrix $\PDGtoFV_{\iel\to\ifaFV} \in \R^{1 \times \nq}$ that maps the high-order DG volume metrics directly to the subcell face.
  By vector algebra, any spatial metric vector can be decomposed uniquely into its scalar magnitude multiplied by its unit direction vector. Applying this identity to Nanson's formula for the physical surface transformation yields
  \begin{align*}
    (\PDGtoFV_{\iel\to\ifaFV} \ \adjJ^{(\iel)}) \normalvecref_\ifaFV
    = \abs{(\PDGtoFV_{\iel\to\ifaFV} \ \adjJ^{(\iel)}) \ \normalvecref_\ifaFV} \normalvec_\ifaFV.
  \end{align*}
  Equating the scalar magnitude to the physical subcell surface area element, $\diagJ^{(\ielFV)}_\ifaFV = \abs{(\PDGtoFV_{\iel\to\ifaFV} \ \adjJ^{(\iel)}) \ \normalvecref_\ifaFV}$, retrieves the target definition in~\cref{eq:FVnormvec}. This completes the proof.
\end{proof}

Thus, the subcell face normal vector $\normalvec_\ifaFV$ and its consistent surface element $\diagJ^{(\ielFV)}_\ifaFV$ represent the integral mean values of the high-order DG face via a conservative projection.
Building upon this geometric foundation, a consistent coupling between DG face and subcell surface fluxes can be formulated.

\begin{lemma}[Conservative Coupling of Subcell Surface Fluxes]
  Let \Cref{lem:subcell_area} hold. The high-order DG surface flux and the subcell FV surface fluxes can be coupled conservatively on arbitrary curvilinear grids. The target DG surface flux is obtained via a conservative projection of the subcell FV fluxes
  \begin{align}
    \diagJ_\ifa^{(\iel)} (\mathbf{f}_\ifa \cdot \normalvec_\ifa) = \PFVtoDG_\ifaFV \left( \diagJ_\ifaFV^{(\ielFV)} (\mathbf{f}_\ifaFV \cdot \normalvec_\ifaFV) \right),
    \label{eq:fvdg_coupling_flux}
  \end{align}
  where $\PFVtoDG_\ifaFV$ is the conservative subcell-to-DG reconstruction operator.
  \label{lem:flux_coupling}
\end{lemma}

\begin{proof}
  To ensure global conservation across the coupled interface, we demand that the total integrated flux over a global DG surface element matches the sum of its FV counterparts over all overlapping subcells
  \begin{align*}
    \int_{\Gamma_\ifa} \mathbf{f}_\ifa \cdot \adjJ^{(\iel)} \normalvecref_\ifa \ d \Gamma = \sum_{\ifaFV=1}^{n_\ifaFV} \int_{\Gamma_\ifaFV} (\mathbf{f}_\ifa \cdot \adjJ^{(\iel)} \ \normalvecref_\ifa)[\Gamma] \ d \Gamma = \sum_{\ifaFV=1}^{n_\ifaFV} \mathbf{f}_\ifaFV \cdot \adjJ^{(\ielFV)} \normalvecref_\ifaFV |\Gamma_\ifaFV|.
  \end{align*}
  We approximate the continuous DG surface quantities over a subcell face $\Gamma_\ifaFV$ by leveraging the polynomial interpolation matrix $\Vm_\ifaFV(\Gamma)$ and the nodal-to-modal projection operator $\Pm_\ifa$
  \begin{align*}
    \int_{\Gamma_\ifaFV} (\mathbf{f}_\ifa \cdot \adjJ^{(\iel)} \ \normalvecref_\ifa)[\Gamma] \ d \Gamma =
      \int_{\Gamma_\ifaFV} \Vm_\ifaFV(\Gamma) \Pm_\ifa (\mathbf{f}_\ifa \cdot \adjJ^{(\iel)} \ \normalvecref_\ifa) \ d\Gamma.
  \end{align*}
  Evaluating this subcell face integral numerically via Gaussian quadrature on the discrete integration nodes $\refpos_\ifaFV$ yields
  \begin{align*}
    \int_{\Gamma_\ifaFV} \Vm_\ifaFV(\Gamma) \Pm_\ifa (\mathbf{f}_\ifa \cdot \adjJ^{(\iel)} \ \normalvecref_\ifa)  d\Gamma =
      \mathbf{1}^\transpose \Wm_\ifaFV \Vm_\ifaFV(\refpos_\ifaFV) \Pm_\ifa \ (\mathbf{f}_\ifa \cdot \adjJ^{(\iel)} \ \normalvecref_\ifa).
  \end{align*}
  Enforcing that this integrated quantity equals the integral mean value evaluated at the subcell face barycenter and dividing by the reference surface area of the subcell face, $|\Gamma_{\ifaFV}|$, yields the consistent row-wise DG-to-FV projection operator $\PDGtoFV_\ifaFV$, matching \cref{eq:PDGtoFVFace}.
  Applying Nanson's formula, given in~\cref{eq:nanson}, then results in
  \begin{align*}
    f_\ifaFV \cdot (\diagJ_\ifaFV^{(\ielFV)} \normalvec_\ifaFV) = \frac{1}{|\Gamma_\ifaFV|} \mathbf{1}^\transpose \Wm_\ifaFV \Vm_\ifaFV(\refpos_\ifaFV) \Pm_\ifa \ \left( \mathbf{f}_\ifa \cdot \left(\diagJ_\ifa^{(\iel)} \normalvec_\ifa \right) \right) = \PDGtoFV_\ifaFV \left( \mathbf{f}_\ifa \cdot \left(\diagJ_\ifa^{(\iel)} \normalvec_\ifa \right) \right).
  \end{align*}
  Noting that the subcell-to-DG reconstruction matrix $\PFVtoDG_\ifaFV$ is the pseudo-inverse satisfying $\PFVtoDG_\ifaFV \PDGtoFV_\ifaFV = \mathbf{I}$ yields the target coupled expression in~\cref{eq:fvdg_coupling_flux}. This completes the proof.
\end{proof}

With the geometric relation of the interface fluxes established in \Cref{lem:flux_coupling}, the arbitrary fluxes $\mathbf{f} \cdot \normalvec$ are replaced either by a consistent numerical flux functions or the viscous surface fluxes.
An illustrative example of the coupling between both schemes, FV subcell and DGSEM, is shown in \cref{fig:fvdg_coupling_face}.

At shared element boundaries, the conservative coupling between the DG and FV subcell operators is established by first mapping the DG conservative variables at the interface to the FV subgrid using \cref{eq:PDGtoFVFace}. The numerical flux is subsequently evaluated on the subcell level and conservatively mapped back to the DG space via the projection matrix defined in \cref{eq:PFVtoDGFace}.
This yields the modified numerical surface flux for the DG scheme, which replaces the baseline term in~\cref{eq:dgsem:convective}
\begin{align}
  \diagJ_\ifa^{(\iel)} f^\ast(\cons_\ifa^+,\cons_\ifa^-;\normalvec_\ifa) = \PFVtoDG_\ifaFV \diagJ_\ifaFV^{(\ielFV)} f^\ast(\PDGtoFV_\ifaFV \cons_\ifa^{+}, \cons_{\ifaFV}^{-}; \normalvec_\ifaFV) = (\diagJ f^\ast)_\ifa.
  \label{eq:fvdg_coupling_fluxc}
\end{align}

An analogous procedure is applied to evaluate the viscous terms at a hybrid DG/FV element interface. The resulting viscous surface flux for the DG method, which replaces the term $\diagJ^{(\iel)}_\ifa \avg{\fphys_v(\cons_\ifa,\nabla \prim_\ifa) \cdot \normalvec_\ifa}$ in \cref{eq:dgsem:viscous}, is expressed as
\begin{align}
  (\diagJ f_v)_\ifa = \PFVtoDG_\ifaFV \diagJ_\ifaFV^{(\ielFV)} \fphys_v\left(\avg{\bar{\cons}^{(\ielFV),+}, \cons^{(\ielFV),-}}, \ \avg{\widebar{\nabla \prim}^{(\ielFV),+}, \nabla \prim^{(\ielFV),-}}\right) \cdot \normalvec_\ifaFV.
  \label{eq:fvdg_coupling_fluxv}
\end{align}
where the terms inside the operator mapping again correspond to the standard low-order viscous fluxes evaluated on the FV subcells faces.
Here, $\bar{(\cdot)}^{(\ielFV)} = (\PDGtoFV_\ielFV (\cdot)^{(\iel)})$ denotes quantities in the $\iel$-th DG cell mapped to the FV cell averages, and $(\cdot)^{(\ielFV),+}$ represents the integral mean values for the FV sub-cells that directly touch the shared FV/DG interface.

This choice avoids the necessity of evaluating the fluxes at individual quadrature points while guaranteeing that the coupled DG/FV scheme is conservative on arbitrary curvilinear grids.
If the interface does not separate differing frameworks, i.e., a pure DG-to-DG or FV-to-FV element interface, the standard DG or FV surface fluxes are utilized, respectively.
The conservation statement is formalized in~\Cref{thm:conservation}. A numerical demonstration is provided in~\cref{sec:validation}.
Building directly upon~\Cref{thm:conservation}, the coupled scheme is also free-stream preserving on curved meshes, as stated in~\Cref{corr:freestream}.

\begin{figure}
  \centering
  \includegraphics[width=0.9\linewidth]{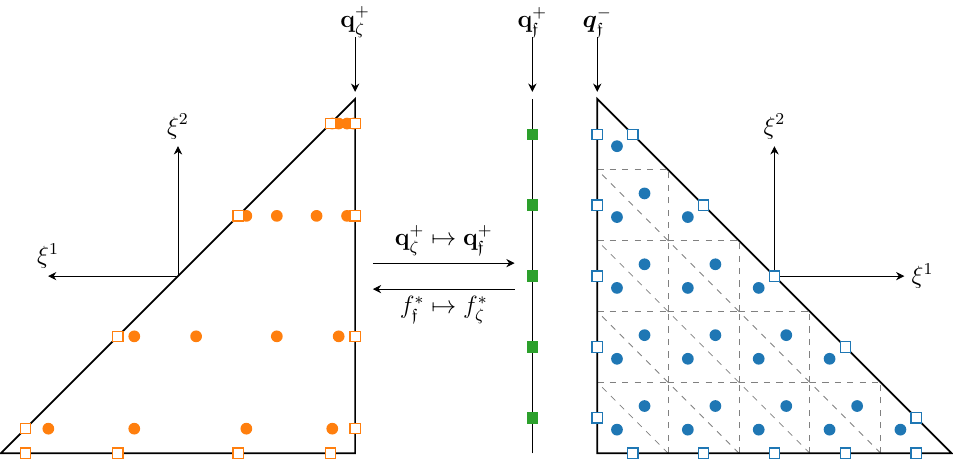}
  \caption{Coupling of DG (left) and FV (right) cell interfaces for triangular elements. The degrees of freedom of the DG scheme are
  highlighted by orange circles and the barycenters of the FV subcells by blue circles. The corresponding degrees of freedom at the
  cell faces are indicated by the squares. The green dots represent the so-called mortar interface used to compute the numerical
fluxes.}
  \label{fig:fvdg_coupling_face}
\end{figure}

\begin{theorem}[Conservation]
  Assume a computational domain with periodic boundary conditions.
  Let the numerical surface fluxes $f^\ast$ be consistent, and assume the volume quadrature rule is of at least degree $2\ppn$ to guarantee the exactness of the mass matrices.
  Assume the metric identities hold ($\divXI \Jm = \Null$) and the normal vectors are computed using~\cref{eq:nanson}, such that $\Dm$ satisfies the discrete metric identities $\Dm ~ \mathbf{1}_{\nq} = \mathbf{0}$.
  Furthermore, let the subcell volume metrics be defined via the conservative volumetric projection, and let the subcell normal vectors be defined using \Cref{lem:subcell_area}.
  If the coupled convective and viscous surface fluxes at hybrid DG/FV element interfaces are computed via the discrete projection mappings defined in \Cref{lem:flux_coupling}~(\cref{eq:fvdg_coupling_fluxc,eq:fvdg_coupling_fluxv}), then the semi-discrete coupled DG/FV scheme is discretely conservative.
  The local conservation properties hold identically for a high-order DG element
  \begin{align}
    \Mm \diagJref \diagJ^{(\iel)} \cons^{(\iel)}_t = - \sum_{\ifa=1}^{\Nf} \Vm_\ifa^\transpose \Wm \diagJref_\ifa \lr{(\diagJ f^\ast)_\ifa - (\diagJ f_v)_\ifa}
  \end{align}
  and analogously for each $\ielFV$-th FV subcell
  \begin{align}
    |E_\ielFV|\diagJ^{(\ielFV)} \cons^{(\ielFV)}_t =
    - \sum_{\ifaFV=1}^{N_{\ifaFV_{\text{inner}}}} \diagJ_\ifaFV^{(\ielFV)} f^\ast(\cons_\ifaFV^{+},\cons_\ifaFV^{-};\normalvec_\ifaFV)
    - \sum_{\ifaFV=1}^{N_{\ifaFV_{\text{outer}}}} \diagJ_\ifaFV^{(\ielFV)} f^\ast(\PDGtoFV_\ifaFV \cons_\ifa^{+}, \cons_{\ifaFV}^{-}; \normalvec_\ifaFV) \\
    + \sum_{\ifaFV=1}^{N_{\ifaFV_{\text{inner}}}} \diagJ_\ifaFV^{(\ielFV)} \fphys_v(\cons_\ifaFV,\nabla \prim_\ifaFV) \cdot \normalvec_\ifaFV
    + \sum_{\ifaFV=1}^{N_{\ifaFV_{\text{outer}}}} \diagJ_\ifaFV^{(\ielFV)} \fphys_v\left(\avg{\bar{\cons}^{(\ielFV),+}, \cons^{(\ielFV),-}}, \avg{\widebar{\nabla \prim}^{(\ielFV),+}, \nabla \prim^{(\ielFV),-}}\right),\nonumber
  \end{align}
  where $N_{\ifaFV_{\text{inner}}}$ and $N_{\ifaFV_{\text{outer}}}$ denote the number of subcell faces residing internally and on the DG element boundary, respectively.
  \label{thm:conservation}
\end{theorem}

\begin{proof}
  We first consider the DG element update. Summing the statement in~\cref{eq:dgsem:viscous} over all degrees of freedom in one element, exploiting the property $\Dm ~ \mathbf{1}_{\nq} = \mathbf{0}$ and the discrete metric identities, the volumetric terms telescope to zero, leading to
  \begin{align*}
    \mathbf{1}_{\nq}^\transpose \Mm \diagJref \diagJ^{(\iel)} \cons^{(\iel)}_t = - \mathbf{1}_{\nq}^\transpose \sum_{\ifa=1}^{\Nf} \Vm_\ifa^\transpose \Wm \diagJref_\ifa \lr{(\diagJ f^\ast)_\ifa - (\diagJ f_v)_\ifa}.
  \end{align*}
  By construction, the coupled boundary terms $(\diagJ f^\ast)_\ifa$ and $(\diagJ f_v)_\ifa$ are defined using the conservative projection operator derived in \Cref{lem:flux_coupling}.
  Applying this lemma ensures that the total flux injected or removed at the DG face matches the algebraic sum of the corresponding individual subcell fluxes appearing in the subcell finite volume expression~(\cref{eq:FV_scheme})
  \begin{align*}
    \mathbf{1}_{\nq}^\transpose \Vm_\ifa^\transpose \Wm \diagJref_\ifa (\diagJ f^\ast)_\ifa = \sum_{\ifaFV=1}^{n_\ifaFV} \diagJ_\ifaFV^{(\ielFV)} f^\ast(\PDGtoFV_\ifaFV \cons_\ifa^{+}, \cons_{\ifaFV}^{-}; \normalvec_\ifaFV) |\Gamma_{\ifaFV}|.
  \end{align*}
  The same applies for the viscous fluxes.
  Local conservation for both schemes is thus established simultaneously, as the projection operators preserve the integral properties across the shared faces.
  Global conservation follows by summing these local cell updates over all elements $\iel = 1, \dots, \nel$ and subcells across the entire computational domain.
  For internal subcell faces, the uniqueness of the numerical fluxes guarantee pairwise cancellation.
  For the hybrid DG/FV interfaces, the projection operator $\PFVtoDG_\ifaFV$ is the exact algebraic adjoint of $\PDGtoFV_\ifaFV$ under the weight-adjusted surface mass matrix metric.
  Consequently, all inter-element fluxes form a telescoping sum and cancel out pairwise under periodic boundary conditions. This completes the proof of \Cref{thm:conservation}.
\end{proof}

\begin{corollary}[Free-Stream Preservation of the Coupled Scheme]
  Let \Cref{thm:conservation} and \Cref{lem:flux_coupling} hold, and let $\Vm_\ifa$ be exact for a constant state. The coupled high-order DG / subcell FV scheme is free-stream preserving on arbitrary curvilinear grids, meaning that an initially uniform constant state $\cons = \mathbf{C}$ is preserved exactly over time, $\cons_t = \mathbf{0}$.
  \label{corr:freestream}
\end{corollary}

\begin{proof}
  We assume a uniform constant state background field $\cons = \mathbf{C}$.
  Under the assumption of a constant state $\cons = \mathbf{C}$, the physical gradients vanish identically, causing the auxiliary lifted viscous gradients to disappear ($\mathbf{g} = \Null$).
  Free-stream preservation is verified by substituting the conservatively coupled numerical flux from~\cref{eq:fvdg_coupling_flux} into the semi-discrete DG surface operator
  \begin{align}
    \Mm \diagJref \diagJ^{(\iel)} \cons^{(\iel)}_t = - \sum_{\ifa=1}^{\Nf} \Vm_\ifa^\transpose \Wm \diagJref_\ifa \PFVtoDG_\ifaFV \left( \diagJ_\ifaFV^{(\ielFV)} f^\ast\left(\PDGtoFV_\ifaFV \mathbf{C}, \mathbf{C}; \normalvec_\ifaFV\right) \right).
  \end{align}
  Since the discrete $L_2$-projection operator is exact for constant states, it holds $\PDGtoFV_\ifaFV \mathbf{C} = \mathbf{C}$, and the consistency of the numerical flux function yields
  \begin{align}
    f^\ast(\mathbf{C}, \mathbf{C}; \normalvec_\ifaFV) = \mathbf{F}_c(\mathbf{C}) \cdot \normalvec_\ifaFV.
  \end{align}
  An identical substitution is applied to the $\ielFV$-th FV subcell operators.
  For internal subcell faces, the consistency of the standard FV flux guarantees cancellation across the subcells.
  For the outer subcell faces, substituting the projected state into the numerical flux results in
  \begin{align}
    f^\ast\left(\PDGtoFV_\ifaFV \cons_\ifa^{+},\cons_\ifaFV^{-};\normalvec_\ifaFV\right) = f^\ast\left(\mathbf{C},\mathbf{C};\normalvec_\ifaFV\right) = \mathbf{F}_c(\mathbf{C}) \cdot \normalvec_\ifaFV,
  \end{align}
  matching the internal flux state.
  Because the discrete operators satisfy both the metric identities and the metric closure property exactly, the geometric terms sum to zero, ensuring that $\cons_t = \mathbf{0}$.
  This completes the proof.
\end{proof}

\subsection{Troubled-cell Detection}
The final essential ingredient for maintaining stability in the presence of strong gradients is a mechanism to distinguish between smooth and non-smooth flow regions. The reliable detection of these "troubled cells" is critical for triggering the localized subcell limiting without compromising the high-order accuracy of the DGSEM. In this work, the troubled-cell indicator is given as
\begin{align}
  \mathcal{I}^{(\iel)} = \max\left( \mathcal{I}_{\text{jump}}(\uind^{(\iel)}), C \ (1-\mathcal{S}^{(\iel)})
  \right)
\end{align}
which is composed of a jump indicator, $\mathcal{I}_{\text{jump}}$, and a sanity check, $\mathcal{S}^{(\iel)}$, with $C \geq \mathcal{I}_{\text{upper}}$ being a large constant number.
If the indicator value is below the lower threshold $\mathcal{I} < \mathcal{I}_{\text{lower}}$, the FV scheme is switched to DG and vice
versa for $\mathcal{I} \geq \mathcal{I}_{\text{upper}}$, cf.~\cref{fig:fv_ind} (left).
The thresholds are bounded such that $\mathcal{I}_{\text{upper}} \geq \mathcal{I}_{\text{lower}} \geq 0$.
The indicator is evaluated based on the nodal DG solution in the element and on the facets, given as
\begin{align*}
  \uind^{(\iel)} = (\uind^{(\iel)}, \{\uind_\ifa \}_{\ifa \in \Nf} )  \in \mathcal{P}^{(\iel)} = \nq \cup \Nf,
\end{align*}
with $\uind$ being the vector of indicator variables, e.g., entropy, including pressure and density.

The positivity-preserving sanity indicator is defined as
\begin{align}
\mathcal{S}^{(\iel)} =
\begin{cases}
  1 &: \rho_i > \epsilon \wedge p_i > \epsilon \quad \forall \, i \in \mathcal{P}^{(\iel)}, \\
  0 &: \text{else},
\end{cases}
\end{align}
where $\epsilon = 10^{-16}$.
For DG cells satisfying this positivity constraint, an oscillation indicator inspired by~\cite{Jameson1981} is utilized, defined as
\begin{align}
  \mathcal{I}_{\text{jump}}(\uind) =
  \frac{1}{|\Omega^{\iel}|} \sum_{n \in \Omega^{(\iel)}}
  \frac{\abs{\check{q}_{\text{min},n} - 2 \check{q}_{n} + \check{q}_{\text{max},n}}}
       {\abs{\check{q}_{\text{min},n} + 2 \check{q}_{n} + \check{q}_{\text{max},n}} + \epsilon} \w_{n} \diagJref \diagJ_{n}^{(\iel)}
\end{align}
with $\w$ being the quadrature weights in the volume and $\check{q}_{\min,n}$ and $\check{q}_{\max,n}$ indicates the mimimum/maximum of $\check{q}_n$ for its three-point stencil along a coordinate line in the reference space, as illustrated in~\cref{fig:fv_ind}
The face values $\uind_\ifa$ are based on a distance-to-mean criterion, $\uind^- \in \Gamma^{(\ifa)}$ and $\uind^+ \in \Gamma^{(\ifa,\text{neighbor})}$, given as
\begin{align}
  \uind_\ifa =
   \begin{cases}
     \uind^- & : \ \abs{\uind^- - \bar{\uind}^{(\iel)}} > \abs{\uind^+ - \bar{\uind}^{(\iel)}}, \\
     \uind^+ & : \ \text{else},
   \end{cases}
 \end{align}
 with the element-local mean $\bar{\uind}^{(\iel)}$.

 \begin{figure}
   \centering
   \begin{subfigure}[c]{0.55\linewidth}
     \includegraphics[width=\textwidth]{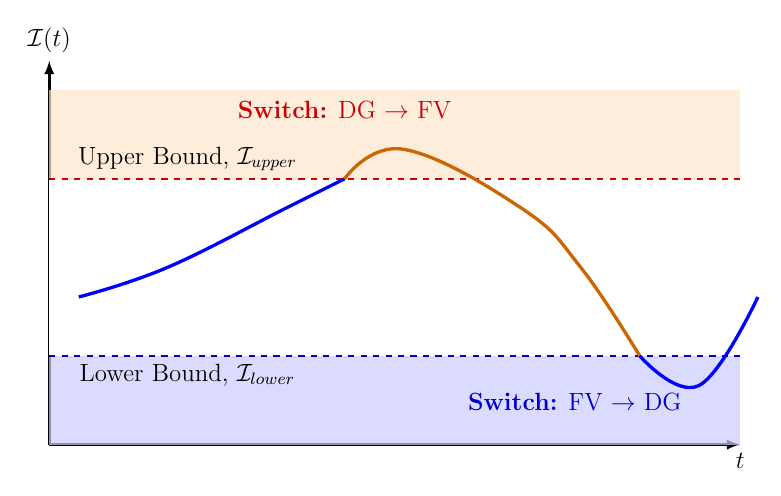}
   \end{subfigure}
   \hfill
   \begin{subfigure}[c]{0.35\linewidth}
     \includegraphics[width=\linewidth]{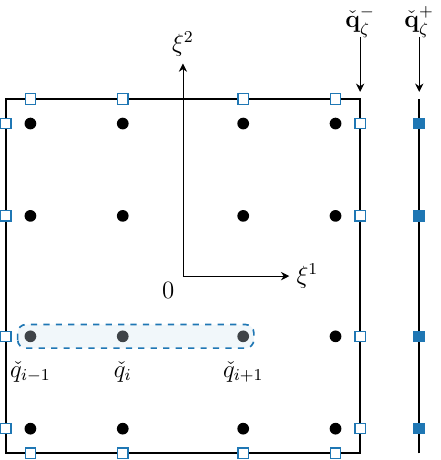}
   \end{subfigure}
   \caption{Left: Exemplary sketch of the switching procedure based on the indicator value over time. Right: Stencil used for the jump indicator.}
   \label{fig:fv_ind}
\end{figure}

\section{Numerical Experiments}
\label{sec:validation}

In the following, the proposed scheme is validated through a series of numerical experiments.
The study begins with the Sedov blast problem to assess the scheme's shock-capturing robustness and verify conservation, followed by a method of manufactured solutions to verify formal convergence rates on curvilinear grids.
The accuracy of the viscous flux discretization and the consistency of the boundary conditions is validated using a lid-driven
cavity.
Finally, the method's capability to handle complex, real-world engineering geometries is demonstrated through the simulation of transonic flow around a NACA 0012 airfoil on a hybrid curved discretization.
A thorough validation of the DGSEM on hybrid, curvilinear meshes can be found in~\cite{Keim2026}.
All following quantities are non-dimensional.
If not state otherwise, $\FVppn = 2\ppn+1$ and a fourth-order low storage Runge--Kutta method \cite{Carpenter1994} are utilized.

The numerical methods presented in this paper are implemented in the open-source high-order solver
FLEXI\footnote{\url{https://github.com/flexi-framework/flexi}}~\cite{Krais2019,Kempf2024}.
All curvilinear meshes are generated using the high-order pre-processor PyHOPE~\cite{Kopper2025pyhope}.

\subsection{Conservation for the 2D and 3D Sedov Blast Problem}
The conservation properties of the scheme are evaluated for the 2D and 3D Sedov blast wave, see, e.g.,~\cite{Zhang2010}, with the initial
conditions given as
\begin{align*}
  \arr{\rho, \rho \vel, \rho e} =
  \begin{cases}
    1, \mathbf{0}, 0.244816 / (\Delta x_{\FV}^d) &: x_i \leq \Delta x_{\FV}, \\
    1, \mathbf{0}, 10^{-12} &: \text{else},
  \end{cases}
\end{align*}
with $\Delta x_{\FV}$ being the average grid size of the FV subcells in each direction.
The solution is advanced until $t=1$ by a low-storage third-order accurate Runge--Kutta scheme with three stages and $\text{CFL} =
0.9$ on a computational domain $\Omega \in [-1.1,1.1]^d$ using $\ppn=4$.
The Rusanov flux function is chosen as a numerical flux.
The lower and upper thresholds of the jump indicator are chosen as $\mathcal{I}_{\text{lower}} = 0.025$ and
$\mathcal{I}_{\text{upper}} = 0.030$, respectively.
The mesh is constructed by initiating the process with a $50^2$ quadrilateral element grid in 2D and a $20^3$ hexahedral element
grid in 3D.
For the former, each quadrilateral element is either maintained as a single quad or divided into two triangles, while for the
latter, each hexahedral element is either maintained as a single hexahedron (\texttt{HEXA}) or divided into two prisms (\texttt{PRIS}), six pyramids (\texttt{PYRA}), or six tetrahedrons
(\texttt{TETRA}), see~\cite{Keim2026}.
In order to circumvent the generation of mixed element interfaces, a pyramid may, if necessary, be divided into two tetrahedral elements.
For the purpose of generating a more challenging test case, the mesh is sinusoidally deformed by perturbing the node positions as follows
\begin{align*}
  \mathrm{2D}:& \ \pos = \left[ \begin{matrix}
  \pos[1] + \epsilon L
  \cos\lr{\nicefrac{\pi}{2L}(\pos[1]-\nicefrac{1}{2})}\cos\lr{\nicefrac{3\pi}{2L}(\pos[2]-\nicefrac{1}{2})}\\
  \pos[2] + \epsilon L
  \sin\lr{\nicefrac{2\pi}{L}(\pos[1]-\nicefrac{1}{2})}\cos\lr{\nicefrac{\pi}{2L}(\pos[2]-\nicefrac{1}{2})}\\
\end{matrix} \right],\\[2ex]
  \mathrm{3D}:& \ \pos = \left[ \begin{matrix}
      \pos[1] + \epsilon L
      \cos\lr{\nicefrac{\pi}{2L}(\pos[1]-\nicefrac{1}{2})}\sin\lr{\nicefrac{2\pi}{L}(\pos[2]-\nicefrac{1}{2})}\cos\lr{\nicefrac{\pi}{2L}(\pos[3]-\nicefrac{1}{2})}\\
      \pos[2] + \epsilon L
      \cos\lr{\nicefrac{3\pi}{2L}(\pos[1]-\nicefrac{1}{2})}\cos\lr{\nicefrac{\pi}{2L}(\pos[2]-\nicefrac{1}{2})}\cos\lr{\nicefrac{\pi}{2L}(\pos[3]-\nicefrac{1}{2})}\\
      \pos[3] + \epsilon L
      \cos\lr{\nicefrac{\pi}{2L}(\pos[1]-\nicefrac{1}{2})}\cos\lr{\nicefrac{\pi}{L}(\pos[2]-\nicefrac{1}{2})}\cos\lr{\nicefrac{\pi}{2L}(\pos[3]-\nicefrac{1}{2})}
  \end{matrix} \right]
\end{align*}
where $L=1$ denotes the length of the domain and $\epsilon$ chosen as $\epsilon=0.15$, and approximated using $\ppngeo=2$.
The subsequent numerical test cases will employ the same procedure for mesh generation.
The density field with the FV subcell distribution is shown in~\cref{fig:validation:sedov} (left) for the 2D configuration and in~\cref{fig:validation:sedov3D} for the 3D configuration.
The errors shown in~\cref{tab:validation:conservation} and defined as
\begin{align*}
  err(\cons) = \max_{t\in T} \left( \int_\Omega \cons(t) d\mathbf{x} - \int_\Omega \cons(t=0) d\mathbf{x} \right),
\end{align*}
demonstrate that the scheme is conservative for the density and the total energy and preserves the symmetry of the test case given the curvature of the mesh, indicated by errors in the integral momentum in the order of machine precision.
The computational results are comparable to the reference data, as illustrated in~\cref{fig:validation:sedov} (right).
The 3D solution is more dissipative then its 2D counterpart due to the lower number of elements used.

\begin{table}
  \centering
  \begin{tabular}{c|ccccc}
& $\rho$ & $\rho \vel[1]$ & $\rho \vel[2]$ & $\rho \vel[3]$ & $\rho e$   \\ \hline
    $d=2$ & 1.180866E-12 & 6.388728E-15 & 5.739526E-14 & 0.000000E+00 & 1.306371E-14  \\ \hline
    $d=3$ & 4.381139E-12 & 7.038317E-15 & 6.036931E-14 & 6.036036E-14 & 1.687538E-14 \\ \hline
  \end{tabular}
  \caption{Conservation and symmetry error of $err(\cons)$ for a 2D and 3D Sedov blast wave after $t=1$.}
  \label{tab:validation:conservation}
\end{table}

\begin{figure}[htbp!]
  \centering
  \includegraphics[width=\textwidth]{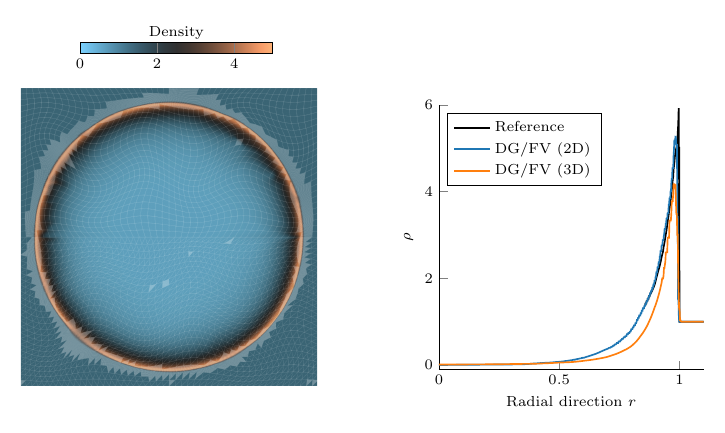}
  \caption{2D Sedov blast at $t=1$ on a curvilinear hybrid mesh using $\ppn=4$ and $\FVppn=2\ppn+1$. Left: Density field
  with the FV subcells highlighted by the subdivided mesh cells. The FV subcell solution is visualized by integral means. Right: Density profile along the radial direction for the 2D and 3D
Sedov blast wave compared to the reference solution, see, e.g.,~\cite{Zhang2010}. The 3D solution is more dissipative due to the lower number of elements used.}
  \label{fig:validation:sedov}
\end{figure}

\begin{figure}[htbp!]
  \centering
  \includegraphics[width=\textwidth]{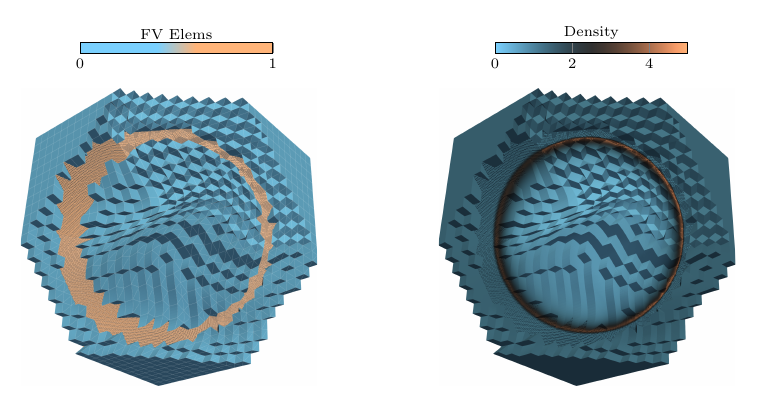}
  \caption{Cutout of the 3D Sedov blast at $t=1$ on a curvilinear hybrid mesh (\texttt{HEXA},\texttt{PRIS},\texttt{PYRA}, and \texttt{TETR}) using
    $\ppn=4$ and $\FVppn = 2\ppn+1$. Left: FV subcells amount. The amount of FV subcells on
  the left looks more due to the cutting and the individually visualized FV subcells. Right: Density field. The FV subcell solution
is represented by integral means.}
  \label{fig:validation:sedov3D}
\end{figure}

\subsection{Experimental Order of Convergence}
The $h$-convergence of the hybrid DG/FV subcell schemes is discussed using the method of manufactured solutions,
cf.~\cite{Gassner2009}. Further details on the $p$-convergence properties of the DGSEM are given in~\cite{Keim2026}.
Following~\cite{Hindenlang2012}, the exact solution is assumed to be of the form
\begin{align}
  \rho = 2+0.1\sin(2\pi(\pos[1]+\pos[2]+\pos[3]-t)), \ \rho \vel = \rho, \ \rho e = (\rho)^2.
\end{align}
The computational domain $\Omega \in [0,1]^d$ is discretized using $4^d$ to $32^d$ elements and $\ppn=2$.
The mesh generation and deformation is similar as for the conservation test.
The timestep is chosen sufficiently small as to have no influence on the discretization error.
A checkerboard indicator and Roe's numerical flux function are utilized.
The chosen error norm is the discrete $L_2$ error of the density.
The results in~\cref{fig:validation:conv} highlight the spatial convergence properties of the DGSE and 2nd-order FV operator for 2D
(dashed line) and 3D (solid line), which both achieve the expected order of convergence under $h$-refinement.
The more then second-order convergence in the 2D case can be attributed to the highly disturbed mesh.

\begin{figure}
  \begin{subfigure}{0.45\linewidth}
    \includegraphics[width=\textwidth]{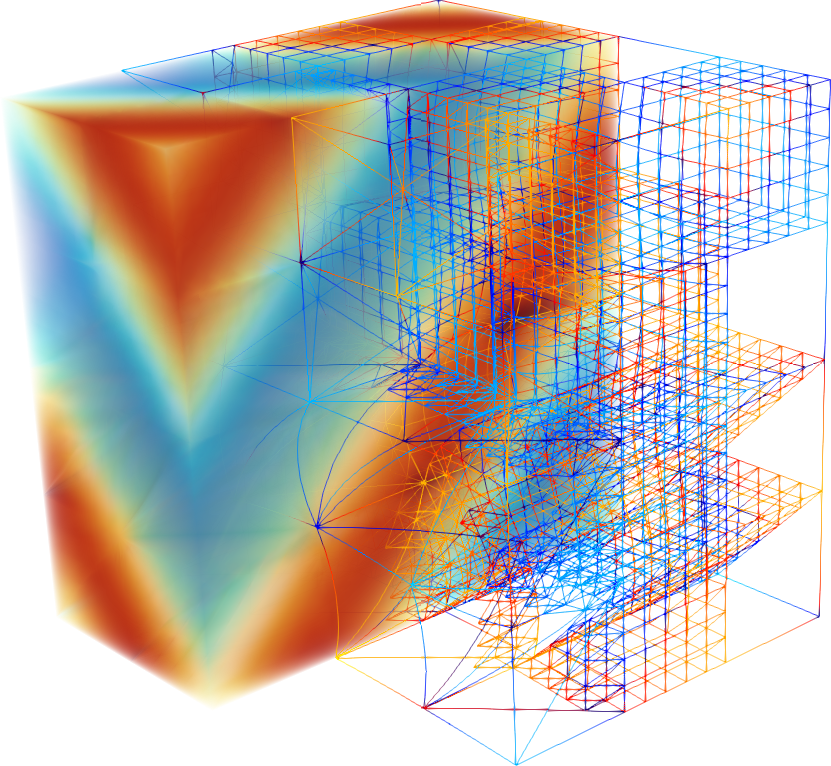}
  \end{subfigure}
  \hfill
  \begin{subfigure}{0.45\linewidth}
    \includegraphics[width=\textwidth]{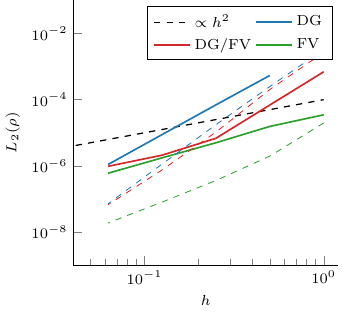}
  \end{subfigure}
  \caption{Left: Visualization of the mesh and density field of the 3D convergence test. Subdivided elements are FV subcells, otherwise
  DG. Right: $h$-convergence for $\ppn=2$ and $\FVppn = 2\ppn+1$ for 2D and 3D. The colored solid lines are the 3D results
and the 2D results are indicated by the colored dashed lines.}
  \label{fig:validation:conv}
\end{figure}

\subsection{Lid-Driven Cavity in 2D and 3D}

The lid-driven cavity is a common benchmark to validate the implementation of the viscous fluxes including boundary conditions.
A Reynolds number of $Re=100$ is chosen based on the initial state $\arr{\rho, \vel, p} = \arr{1,\mathbf{0},71.42857}$, from which
the dynamic viscosity can be retrieved.
This configuration is selected primarily for code validation purposes, as it allows for a fast verification of the viscous
operators.
For all tests, the computational domain $\Omega \in [0,1]^d$ is discretized by $40^2$ elements in 2D and $20^3$ elements in 3D with with $\ppn=3$.
The mesh is build and curved similar to above.
On the left, right, and bottom, stationary, adiabatic, no-slip wall boundary conditions are applied.
On the top, a wall moving with velocity $\vel=\arr{1,0,0}$ is utilized.
In the case of $d=3$, stationary, adiabatic, no-slip wall boundary conditions are used in the $z$-direction.
A checkerboard indicator and Rusanov's numerical flux function are utilized.
The results for the lid-driven cavity at $t=5$ are depicted in~\cref{fig:validation:cavity}, demonstrating that for the chosen grid
resolution the numerical results are in reasonable agreement to the reference data~\cite{Agarwal1981}.
The small oscillations in the velocity profile can be attributed to the disturbed curvilinear mesh.

\begin{figure}[htbp!]
  \centering
  \includegraphics[width=\textwidth]{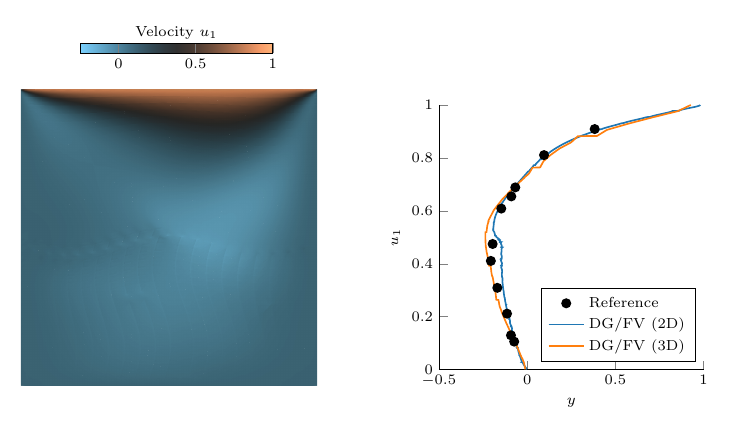}
  \caption{Lid-driven cavity at $Re=100$ and $t=5$. Left: Visualization of a slice in the $xy$-plane of the velocity in $x$-direction at
  $z=0.5$. Right: Velocity in $x$-direction plotted along the $y$-direction at $z=x=0.5$ compared to the reference data~\cite{Agarwal1981}.}
  \label{fig:validation:cavity}
\end{figure}

\subsection{Transonic Flow Past the NACA 0012}

The applicability of the proposed scheme to more complex problems is demonstrated using the flow around the NACA 0012
airfoil.
The initial conditions are chosen as $\rho = 1.0, \ \vel = \arr{0.998, 0.061, 0.0}, \ p=1.340$ (non-dimensional), resulting in a Mach number of $M=0.73$.
The computational domain is discretized using \num{1168} curvilinear prisms and \num{12032} curvilinear hexahedrons with $\ppngeo=2$.
Farfield boundary conditions are employed at the outer boundaries of the computational domain to simulate the undisturbed ambient
environment, without perturbing the viscid flow problem.
Periodic boundary conditions are applied in spanwise direction and adiabatic, slip conditions are prescribed on the airfoil
surface.
A Roe-type dissipation at the element faces are utilized.
The lower and upper thresholds of the jump indicator are chosen as $\mathcal{I}_{\text{lower}} = 0.005$ and
$\mathcal{I}_{\text{upper}} = 0.010$, respectively.
A slice of the second-order unstructured grid together with the instantaneous distribution of the Mach number at
$t=12$ using $\ppn=4$ are depicted in~\cref{fig:naca0012}.

\begin{figure}
  \includegraphics[width=\textwidth]{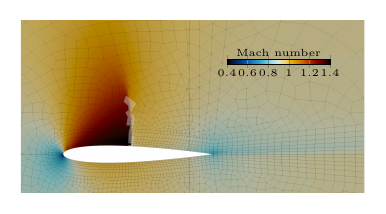}
  \caption{NACA 0012 airfoil. Instantaneous distribution of Mach number at $t=12$ with the second-order curvilinear unstructured
  volume grid and the FV subcells highlighted in white.}
  \label{fig:naca0012}
\end{figure}

\section{Conclusion}
\label{sec:conclusion}

In this work, we presented a robust, efficient, and provably conservative shock-capturing approach for DGSEM on mixed, curvilinear meshes based on a 2nd-order $h$-adaptive FV subcell scheme.
While prior subcell limiting frameworks offer conservation on uniform meshes, our approach extends the $h$-adaptive subcell strategy to curvilinear hybrid meshes, including hexahedra, prisms, tetrahedra, and pyramids, and provides the analytical and numerical proof of its discrete conservation.
The scheme demonstrates excellent stability and convergence across diverse test cases without the sub-element resolution loss typical of artificial viscosity.
Ultimately, this methodology offers a robust pathway for the high-fidelity analysis of shock-laden turbulent flows in realistic engineering geometries, seamlessly balancing geometric flexibility with numerical stability and efficiency.

\begin{acknowledgement}
This work was funded by the European Union and has received funding from the European High Performance Computing Joint Undertaking (JU) and Sweden, Germany, Spain, Greece, and Denmark under grant agreement No 101093393.
The research presented in this paper was funded in parts by Deutsche Forschungsgemeinschaft (DFG, German Research
Foundation) through SPP 2410 Hyperbolic Balance Laws in Fluid Mechanics: Complexity, Scales, Randomness (CoScaRa) and by the state of Baden-Württemberg under the project Aerospace 2050 MWK32-7531-49/13/7 "QUASAR".
\end{acknowledgement}
\ethics{Competing Interests}{The authors have no conflicts of interest to declare that are relevant to the content of this chapter.}

\eject

\bibliographystyle{spmpsci}
\bibliography{references}
\end{document}